\documentclass[12pt]{article}
\usepackage{amsmath}
\usepackage{amsfonts}
\usepackage{mathrsfs}
\usepackage{amssymb}
\usepackage{amsthm}
\usepackage{bigints}
\usepackage{amstext}
\usepackage{color}
\usepackage{graphicx}
\usepackage{xspace}
\usepackage[backend=biber, style=nature,sorting=nty]{biblatex}
\usepackage{color}

\newtheorem{theorem}{Theorem}

\newtheorem{corollary}[theorem]{Corollary}

\newtheorem{example}[theorem]{Example}

\newtheorem{proposition}[theorem]{Proposition}

\def\A{\mathcal{A}}

\def\H{\mathbb{H}}
\def\C{\mathbb{C}}

\def\R{\mathbb{R}}

\def\Sc{\mbox{Sc}\,}
\def\Vec{\mbox{Vec}\,}

\def\iff{\Leftrightarrow}

\def\div{\mbox{div}\,}

\def\curl{\mbox{curl}\,}
\def\H{\mathbb{H}}
\def\B{\mathbb{B}}
\def\C{\mathbb{C}}

\def\R{\mathbb{R}}

\def\Sc{\mbox{Sc}\,}
\def\Vec{\mbox{Vec}\,}

\def\iff{\Leftrightarrow}

\let\epsilon\varepsilon

\let\phi\varphi

\newcommand{\conj}[1]{\overline{#1}}
\newcommand{\inner}[1]{\left\langle #1 \right\rangle}
\newcommand{\norm}[1]{\left|\left| #1 \right|\right|}
\newcommand{\fourier}[1]{\mathcal{F}\left( #1 \right)}
\newcommand{\fourierinv}[1]{\mathcal{F}^{-1}\left( #1 \right)}
\newcommand{\expo}[1]{\exp\left( #1\right)}

\newcommand{\Twm}[1]{T_{\mathscr{C},-}\left[ #1\right]}
\newcommand{\Real}[1]{\text{Re}\left( #1 \right)}
\newcommand{\Imag}[1]{\text{Im}\left( #1 \right)}

\begin{document}

\title{A biquaternionic reformulation of Maxwell's equations via Fourier analysis}
\makeatletter
\author{Aarón Guillén-Villalobos \and Briceyda B. Delgado, \and Héctor Vargas Rodríguez}

\date{}
\maketitle
\begin{abstract}
    We analyze the parabolic Dirac operator $D \pm i\partial_t$ in a biquaternionic setting, characterizing its kernel via generalized div-curl systems and Cauchy-Riemann-type relations between the real and imaginary parts. Using the machinery provided by the Fourier transform, we fully characterize the Fourier transform of a fundamental solution of this operator and construct a well-defined right inverse operator. As an application, we derive explicit vectorial solutions to the time-dependent Maxwell system, extending prior biquaternionic approaches. These tools offer analytical efficiency for complex electromagnetic problems.

\end{abstract}
\mbox{}operators.
\mbox{}

\section{Introduction}
In this work, we employ the following factorization of the wave operator 

\begin{equation*}
    -\Delta+\partial_{tt}=(D+i\partial_t)(D-i\partial_t)=(D-i\partial_t)(D+i\partial_t),
\end{equation*}

where $i$ is the imaginary unit and $D$ is the Moisil-Teodorescu operator. We work in a biquaternionic setting and follow closely the previous results documented in the monograph \cite{kravchenko2003applied}. Other first-order perturbed Dirac operators have been reported in the literature for analyzing and factorizing the wave equation within Clifford algebra frameworks; see, for instance, \cite{eelbode2003solutions, eelbode2004inverse, eelbode2005fundamental, marmolejo2023basic}. In \cite{kravchenko1998kernel}, the authors showed that any solution of the Klein-Gordon equation can be represented through two different solutions of parabolic Dirac equations with the same mass, and in \cite{cerejeiras2008factorization}, the authors studied a general scheme for factorizing second-order time-dependent operators of mathematical physics. In both works, the key point is the reduction of second-order equations into first-order biquaternionic equations.

In this manuscript, we develop fundamental properties of the parabolic Dirac operators $D\pm i \partial_t$. Instead of seeking a fundamental solution $\Phi_{\pm}$ of $D\pm i \partial_t$, we fully characterize its Fourier transform through the exponential quaternionic function. Indeed, 

\begin{align*}
    \fourier{\Phi_\pm}(k,t)=\expo{\mp 2\pi kt}=\cos\left(2\pi|k|t\right)\mp\frac{k}{|k|}\sin\left(2\pi|k|t\right),
\end{align*}

for all $(k,t)\in \mathbb{R}^3\times \mathbb{R}$. Moreover, using properties of the Fourier transform and the quaternionic exponential function, we propose a right inverse operator for these parabolic Dirac operators. As an application, we analyze the time-dependent Maxwell system. Quaternionic and biquaternionic analysis has been previously applied to such physical systems. For the time-harmonic Maxwell system, a series of works \cite{Kravchenko1992, DelKrav2019, delgado2023biquaternionic} adopt a biquaternionic approach. In particular, \cite{DelKrav2019, delgado2023biquaternionic} employ classical integral operators in quaternionic analysis, along with a completion monogenic process that generalizes the classical harmonic conjugate concept from complex analysis.

The Maxwell equations hold importance across many fields of science and applied mathematics. For example, \cite{mitskievich2008classification,mitskievich2006electromagnetic} use invariants of the electromagnetic $2$-form to classify electromagnetic fields and study their propagation, while \cite{vargas2020relativistic} examines stationary electromagnetic fields with angular momentum via the Poynting vector.

Fourier analysis serves as a powerful tool not only in physics but also in digital image processing, signal processing, and cryptography. For instance, \cite{danchin2005fourier} applies it as an alternative method for solving the heat equation by replacing the spatial derivatives with Fourier-space variables, thereby transforming the partial differential equation into a simpler linear differential equation. In \cite{andrade2023maxwell}, the authors use the Fourier transform to recast Maxwell's equations into a system of dot and cross products, which changes both the space and time partial derivatives to the Fourier space variables. 

The outline of the paper is as follows. In Section \ref{sec:preliminaries}, we recall definitions of biquaternionic algebra along with basic properties of the Dirac operator $D$ and the Moisil--Teodorescu system. In Section \ref{sec:parabolic-Dirac-operator}, we derive preliminary results characterizing the kernel of $(D+i\partial_t)$. Proposition \ref{equivalent:system:proposition} establishes that the kernel of $D\pm i\partial_t$ corresponds to a generalized div-curl system: $w=u+iv=(u_0+\vec{u})+i(v_0+\vec{v})\in \text{Ker}(D\pm i\partial_t)$ if and only if
\begin{equation*}
    \begin{array}{rcl}
        -\div(\vec{u})&=&\pm \partial_t v_0,\\
        \nabla u_0+\curl(\vec{u})&=& \pm\partial_t \vec{v},\\
        -\div(\vec{v})&=&\mp\partial_t u_0,\\
        \nabla v_0+\curl(\vec{v})&=&\mp\partial_t \vec{u}.
    \end{array}
\end{equation*}

Meanwhile, Proposition \ref{prop:completation} shows that if $u$ is a complex-valued function satisfying the wave equation, we can construct a quaternion-valued function $v$ such that $u\pm iv$ solves $(D\pm i\partial_t)$. In Section \ref{sec:inverse_dirac_operator}, we construct a right inverse operator for $D+i\partial_t$ (see Theorem \ref{th:right-inverse}) and prove it is well-defined on $L^2(\mathscr{C},\mathbb{B})$. Finally, in Section \ref{sec:rel_maxwell_equations}, we apply the right inverse of the perturbed Dirac operators $D\pm i\partial_t$ to solve the time-dependent Maxwell system in differential form and Cartesian coordinates. In Subsection \ref{subsec:maxwell_homogeneous}, we generalize the results obtained in previous sections. We do a similar analysis for the operator $D\pm i\lambda \partial_t$, where $\lambda\in \mathbb{R}$, and we provide a general solution of Maxwell's equations in a homogeneous and isotropic medium.

\section{Preliminaries}\label{subsec:Preliminaries}\label{sec:preliminaries}
Let us consider the algebra of biquaternions and quaternions, denoted by $\B$ and $\mathbb{H}$, respectively. Let $w_1,w_2\in\B$, we will denote by $u_i, v_i\in \mathbb{H}$ to their respective real and imaginary parts, $i=1,2$. That is,

\begin{equation*}
    w_1=u_1+iv_1, \quad w_2=u_2+iv_2.
\end{equation*}

Following the classical complex rule multiplication, then 

\begin{align*}
    w_1w_2&=(u_1+iv_1)(u_2+iv_2)=u_1u_2-v_1v_2+i(u_1v_2+v_1u_2),
\end{align*}

where $u_1u_2,v_1v_2,u_1v_2$ and $v_1u_2$ follow the classical quaternionic multiplication.

Let us denote the scalar part and vector part of $u$ by $u_0=\Sc{u}$ and $\vec u=\Vec{u}$, respectively. Let $w=u+iv\in\B$, then 

\begin{equation*}
    \begin{array}{rcccl}
        \Sc(w)&=&\Sc(u)+i\Sc(v)&=& u_0+iv_0,\\
        \Vec(w)&=&\Vec(u)+i\Vec(v)&=&\vec{u}+i\vec{v}.
    \end{array}
\end{equation*}

Now, let us denote by $\conj{w}$ the quaternionic conjugation, and $w^*$ stands for the usual complex conjugation. Indeed,
\begin{equation*}
    \begin{array}{rcl}
        \conj{w}=(u_{0}-\vec{u})+i(v_{0}-\vec{v})=(u_0+iv_0)-(\vec u+i\vec v), \quad w^*=(u_{0}+\vec{u})-i(v_{0}+\vec{v}).
    \end{array}
\end{equation*}

From now on, $\Omega\subseteq \R^3$ will be a bounded domain. We will denote the space-time domain by $\mathscr{C}:=\Omega\times (0,\infty)$, $C^{2,2}(\mathscr{C},\mathcal{B})$ will be the space of functions with continuous second-order spatial and temporal partial derivatives defined in $\mathscr{C}$ and taking values on $\mathcal{B}\in \{\mathbb{R}, \mathbb{C},\mathbb{R}^3, \mathbb{C}^3,\H, \mathbb{B}\} $, and $C^{1,1}(\mathscr{C},\mathcal{B})$ will denote the space of functions whose first-order spatial and temporal partial derivatives are continuous. Throughout this work, we will work with biquaternionic-valued functions. That is, $w:\mathscr{C}\to\B$, and $w$ has the form
\begin{equation}\label{eq:real+imaginary}
    w=\sum_{k=0}^3 e_iw_i= \underbrace{u}_{\Real{w}}+i\underbrace{v}_{\Imag{w}},
\end{equation}
where $w_i:\mathscr{C}\to\C$, $i=0,1,2,3$, are complex-valued functions, and $u, v: \mathscr{C}\to\H$ are quaternionic-valued functions.

Recall that the definition of a $L^p(\A, \mathcal{B} )$ is defined as all the set of functions $w$ with domain $\A$ and taking values on $\mathcal{B}$, where $\mathcal{A}\in \{\mathscr{C}, \R^3\times(0,\infty)\}$ and $\mathcal{B}\in \{\mathbb{R}, \mathbb{C},\mathbb{R}^3, \mathbb{C}^3,\H, \mathbb{B}\} $, such that every component of $w$ is integrable over $\A$, for $p$ in $[1,\infty)$, also recall the Sobolev space $W^{k,p}(\A, \mathcal{B})$ defined as the subspace of all the integrable functions in $L^p(\A,\mathcal{B})$ such that every component has $k$ weak derivatives, for $k$ a positive integer. Now, we define the \textit{Moisil--Teodorescu differential operator} as follows

\begin{equation}\label{eq:Dirac-operator}
	D=\sum_{i=1}^3e_i\partial_i,
\end{equation}
where $\partial_i$ represents the partial derivative with respect to the variable $x_i$, for each $i=1,2,3$. In the literature, operator \eqref{eq:Dirac-operator} is also known as the \textit{Dirac operator}.

If the operator $D$ acts over functions in the set $C^{1,1}(\Omega, \B)$, which is the set of $C^{1,1}$-functions defined on $\Omega$ and taking values in $\B$, then

\begin{equation}\label{eq:left-right-Dirac}
    Dw=\sum_{i=1}^3 e_i \partial_i w, 
\end{equation}

for every $w\in C^{1,1}(\Omega, \B)$. In particular, when $w=u+iv \in C^{1,1}(\Omega, \B)$, then \eqref{eq:left-right-Dirac} reduces to

\begin{equation}\label{eq:deco-div-curl}
    Dw=Du +iDv,
\end{equation}

where the action over every quaternion-valued functions $u=u_0+\vec{u}$ and $v=v_0+\vec{v}$ is given by the classical rule

\begin{equation}\label{eq:action-D}
   Du=-\div\vec{u}+\nabla u_0+\curl\vec{u}, \qquad Dv=-\div\vec{v}+\nabla v_0+\curl\vec{v}.
\end{equation}

Notice that $Du, Dv$ are expressed in terms of the gradient $\nabla$, the divergence $\div$, and the curl (or rotational $\curl$). This particular form of decomposing the Dirac operator for biquaternionic-valued functions will be useful in the development of the present work.

A function $w\in C^{1,1}(\Omega,\B)$ is called \textit{left-monogenic} in $\Omega$ when $Dw=0$ in $\Omega$. Respectively, it is called \textit{right-monogenic} when $wD=0$ in $\Omega$. Some basic properties of the Moisil-Teodorescu operator and left-monogenic functions can be found in references \cite{brackx1982, kravchenko2003applied, GuHaSpr2008}. By \eqref{eq:deco-div-curl}-\eqref{eq:action-D}, we can observe that $w:\Omega\to\B$ is left-monogenic in $\Omega$ if and only if $u, v$ simultaneously satisfy in $\Omega$

\begin{equation}\label{eq:div-curl-monogenic}
    \begin{array}{rclcrcl}
        \div \vec u&=&0,& &\div\vec v&=&0\\
        \curl\vec u&=&-\nabla u_0 & & \curl\vec v&=&-\nabla v_0
    \end{array}
\end{equation}
The system \eqref{eq:div-curl-monogenic} is called the \textit{Moisil--Teodorescu system}. It is worth mentioning that it can be considered as the most natural generalization of the Cauchy--Riemann system in three dimensions \cite{kravchenko2003applied}.

Let us consider the \textit{Cauchy kernel}
\begin{equation}\label{eq:Cauchy}
    E(x)=\frac{-x}{4\pi|x|^3},\qquad x\in\mathbb{R}^3\setminus\{0\}.
\end{equation}

It is well-known that the Cauchy kernel is a fundamental solution of the Moisil-Teodorescu operator $D$. For each $w\in L^p(\Omega,\B)$, the \textit{Teodorescu transform} of $w$ is defined as
\begin{equation}\label{eq:Teodorescu}
    T_\Omega[w](x)=-\int_{\Omega} E(y-x)w(y)\, dy\qquad x\in\mathbb{R}^3,
\end{equation}
where $E$ is the Cauchy kernel defined in \eqref{eq:Cauchy}.
In fact, the Teodorescu transform is a right inverse of $D$ in $L^p(\Omega,\B)$, for $p>1$. More precisely, $DT_{\Omega}[w]=w$, for every $w\in L^p(\Omega,\B)$.

We will consider the \textit{wave operator}, $-\Delta+\partial_{tt}$, acting over functions in $C^{2,2}(\mathscr{C})$, it acts naturally component-wise. We say that $w$ satisfies the homogeneous wave equation if

\begin{equation}\label{eq:wave-equation}
    (-\Delta+\partial_{tt})w=0, \quad \text{ in }\Omega.
\end{equation}

\subsection{Background on the Fourier transform}
Let $f\in L^1(\mathbb{R}^3)$. Following \cite{lieb2001analysis}, we define the Fourier transform of a function $f:\R^3\to\B$, denoted by $\fourier{f}$ as follows

\begin{align}\label{eq:Fourier}
    \fourier{f}(k)&=\int_{\R^3} \exp(-2\pi i\inner{k,x})f(x)\, dx=\sum_{i=0}^3 e_i\int_{\R^3} \exp(-2\pi i\inner{k,x})f_i(x)\, dx, 
\end{align}

for all $k\in \mathbb{R}^3$, where

\begin{align*}
    \inner{k,x}&=\sum_{i=1}^3 k_ix_i, \qquad \text{ and }\qquad f=\sum_{i=0}^3 f_ie_i,
\end{align*}

where $f_i:\R^3\to \mathbb{C}$, $i=0, 1, 2, 3.$ We usually work with functions defined in $\mathscr{C}=\Omega \times (0,\infty)$, by convention we will extend the function $f$ by defining $f(x,t)=0$ for all $x$ in $\mathbb{R}^3\setminus \overline{\Omega}$. Despite being functions defined in a subset of $\mathbb{R}^4$, we define the Fourier transform of $f$ at the point $(k,t)\in \mathscr{C}$ as in \eqref{eq:Fourier}. Indeed,
\begin{equation}\label{eq:Fourier-space-time}
    \fourier{f}(k,t):=\int_{\R^3}\exp(-2\pi i\inner{k,x})f(x,t)\,dx.
\end{equation}
Some fundamental properties of the Fourier transform are collected below. 
In particular, the map $f \mapsto \fourier{f}$ is invertible, with inverse 
operator denoted by $\fourierinv{f}$.

\begin{theorem}{\cite[Th.~5.5]{lieb2001analysis}}\label{theorem-1}
    For $f\in L^2(\R^3)$ we define 
    
    \begin{equation}\label{eq:definition-Fourier-inverse}
        \fourierinv{f}(x):=\fourier{f}(-x).
    \end{equation}
    
    Then
    
    \begin{equation}\label{eq:inversos}
        f=\fourierinv{\fourier{f}}.
    \end{equation}
\end{theorem}

Let $f,g\in L^1(\R^3)$, the convolution of $f$ and $g$, denoted by $f*g$, is defined as follows

\begin{equation*}
    (f\ast g)(x)=\int_{\R^3} f(x-y)g(y)\, dy.
\end{equation*}
Most of the time, we will work with convolutions in $\mathbb{R}^3$ unless otherwise specified; for instance, we will denote by $f\ast_{(x,t)} g$ the convolution in $\mathbb{R}^4$.
One important property of the Fourier transform is that it converts differentiation and convolution into multiplication operations, indeed
\begin{theorem}{\cite[Th.\ 5.8]{lieb2001analysis}}\label{theorem-2}
    Let $f\in L^p(\R^3)$ and $g\in L^q(\R^3)$ and let $1+\tfrac{1}{r}=\frac{1}{p}+\frac{1}{q}$. Suppose that $1\leq p,q$ and $r\leq 2$, then
    
    \begin{equation}\label{eq:convolution}
        \fourier{f\ast g}(k)=\fourier{f}(k)\fourier{g}(k).
    \end{equation}    
\end{theorem}

\begin{theorem}{\cite[Th. ~7.9]{lieb2001analysis}}\label{theorem-3}
    Let $f$ be an $L^2(\R^3)$ function with Fourier transform $\fourier{f}$. Then $f$ is in $W^{1,2}(\R^3)$ if and only if the function $k\to |k|\fourier{f}(k)$ is in $L^2(\R^3)$. If it is in $L^2(\R^3)$, then
    \begin{equation}\label{eq:action-gradient}
        \fourier{\nabla f}(k)=2\pi ik\fourier{f}(k),
    \end{equation}
    
    Moreover,
    
    \begin{equation*}
        \norm{f}^2_{W^{1,2}(\R^3)}=\int_{\R^3}|\fourier{f}(k)|^2(1+4\pi^2|k|^2)\,dk.
    \end{equation*}
\end{theorem}

It is worth mentioning that Theorems \ref{theorem-1}-\ref{theorem-3} are enunciated for $\R^3$, but indeed they are valid for $\R^n$ in general.

\section{The parabolic Dirac operator}\label{sec:parabolic-Dirac-operator}
In this section, we are going to analyze the following \textit{parabolic Dirac operators} in a biquaternionic framework

\begin{equation}\label{eq:parabolic-Dirac}
    D\pm i\partial_t,
\end{equation}
where $D$ is the Dirac operator defined previously in \eqref{eq:Dirac-operator} and $i$ is the imaginary unit in $\C$. Following \cite{kravchenko1998kernel, kravchenko2003applied}, we employ the following factorization of the wave operator in terms of the parabolic Dirac operators $D\pm i\partial_t$
\begin{equation}\label{eq:factorization-wave}
        -\Delta+\partial_{tt}=(D+i\partial_t)(D-i\partial_t)=(D-i\partial_t)(D+i\partial_t).
\end{equation}
The above factorization will be the cornerstone of our work. We proceed now to investigate the kernel of $D\pm i\partial_t$. Consider a function $w:\mathscr{C}\to \B$, which belongs to the kernel of the parabolic Dirac operator $D\pm i\partial_t$ in $\mathscr{C}$. As a consequence of the factorization \eqref{eq:factorization-wave}, we can easily see that $w$ is a solution of the wave equation \eqref{eq:wave-equation} component-wise in $\Omega$. As a consequence, it becomes important to have a complete characterization of $\text{Ker }(D\pm i\partial_t)$. From now on, we are going to use the convention for biquaternionic-valued functions \eqref{eq:real+imaginary}. Thus, $w\in \text{Ker }(D\pm i\partial_t)$ if and only if $w$ satisfies a parabolic div-curl system as follows

\begin{proposition}\label{equivalent:system:proposition}
    Let $w:\mathscr{C}\to\B$ be a $C^{1,1}(\mathscr{C},\mathbb{B})$ function. Then $(D\pm i\partial_t)w=0$ in $\mathscr{C}$ if and only if
    
    \begin{equation}\label{eq:div-curl-parabolic}
        \begin{array}{rlc}
                -\div \vec{u}&=&\pm\partial_t v_0,\\
                \nabla u_0+\curl \vec{u}&=&\pm\partial_t\vec{v},\\
                -\div \vec{v}&=&\mp\partial_t u_0,\\
                \nabla v_0+\curl \vec{v}&=&\mp\partial_t\vec{u},\\
        \end{array}
    \end{equation}
    
    where $w=u+iv=(u_0+\vec{u})+i(v_0+\vec{v})$.
\end{proposition}

\begin{proof}
    Let $w:\mathscr{C}\to\B$ be a continuously differentiable function, then
    \begin{align*}
        (D\pm i\partial_t)w=0&\iff
        Du\mp \partial_tv+i(Dv\pm \partial_tu)=0.
    \end{align*}
    Therefore,
    \begin{equation}\label{eq:cauchy-style-equation}
        \left\{
            \begin{array}{rcl}
                Du\mp \partial_tv&=&0,\\
                Dv\pm \partial_tu&=&0.
            \end{array}
        \right.
    \end{equation}
    Using the action \eqref{eq:action-D} of the Dirac operator $D$ and separating the scalar and the vector parts, we can directly verify that \eqref{eq:cauchy-style-equation} and \eqref{eq:div-curl-parabolic} are equivalent systems.
\end{proof}    

Given the above equations, we can see a certain similarity with the Cauchy-Riemann equations, so we want to generalize the harmonic conjugation process. Based on the equivalence between $\text{Ker }(D\pm i\partial_t)$ and system \eqref{eq:cauchy-style-equation} and the properties of the Teodorescu operator $T_{\Omega}$ defined in \eqref{eq:Teodorescu}, we are able to construct the imaginary part of solutions in the kernel of $D\pm i\partial_t$ when the real part is known, and it is a solution of the wave equation \eqref{eq:wave-equation}. That is, 

\begin{proposition}\label{prop:completation}
    Let $u:\mathscr{C}\to\H$ be a $C^{2,2}(\mathscr{C},\mathbb{H})$ function and $u\in\text{Ker }(-\Delta+\partial_{tt}).$ Then we can construct $v:\mathscr{C}\to\H$ such that $u\pm iv\in\text{Ker }(D\pm i\partial_t)$ as follows
    \begin{equation}\label{Riemann:Cauchy:Style:equation}
        v(x,t):=\int_0^t Du(x,s)\,ds-T_\Omega[\partial_t u](x,0)+\nabla\phi(x),
    \end{equation}
    where $T_{\Omega}$ is the Teodorescu operator and $\phi$ is any harmonic function independent of time.
\end{proposition}

\begin{proof}
    Let $u$ be a solution for the wave equation \eqref{eq:wave-equation}. By the proof of Proposition \ref{equivalent:system:proposition} we only need to verify that \eqref{eq:cauchy-style-equation} holds, then
    \begin{align*}
        \pm\partial_tv(x,t)&=\partial_t\left(\int_0^tDu(x,s)\, ds-T_\Omega[\partial_t u](x,0)+\nabla\phi(x)\right)\\
        &= Du(x,t).
    \end{align*}
    Now, using that $DT_{\Omega}=I$ in $\Omega$, we get that
    \begin{align*}
        Dv(x,t)&=\pm D\left(\int_0^tDu(x,s)\, ds-T_\Omega[\partial_tu](x,0)+\nabla\phi(x)\right)\\
        &=\pm \left(\int_0^tD^2 u(x,s)\, ds-DT_\Omega[\partial_tu](x,0)-\Delta \phi(x)\right)\\
        &=\pm \left(\int_0^t-\Delta u(x,s)\,ds-\partial_tu(x,0)\right)\\
        &=\pm \left(\int_0^t-\partial_s^2u(x,s)\,ds-\partial_tu(x,0)\right)\\
        &=\pm(\partial_tu(x,0)-\partial_tu(x,t)-\partial_tu(x,0))\\
        &=\mp\partial_tu(x,t).
    \end{align*}
    Therefore, $w:=u\pm iv$ is a solution of $(D\pm i\partial_t)w=0$ as we desired.
\end{proof}

\begin{example}
    Let $u(x,t)=x+t$, then $(-\Delta+\partial_{tt})u=0$.
    Using that $T_{\Omega}[1]=-\frac{1}{3}x$ (see \cite[Appendix]{gurlebeck1997quaternionic}) and by \eqref{Riemann:Cauchy:Style:equation}, then we can calculate a parabolic metaharmonic conjugates of $u$, namely $v$ as follows
    \begin{align*}
        v(x,t)&= \int_0^t Du(x,s)\, ds-T_\Omega[\partial_t u](x,0)+\nabla \phi(x)\\
        &=\int_0^t -3\,ds-T_\Omega[1](x,0)+\nabla \varphi(x)\\
        &= \left(-3t-T_\Omega[1](x,0)+\nabla \varphi(x)\right)\\
        &= \left(-3t+\frac{1}{3}x+\nabla \varphi(x)\right).
    \end{align*}
    Therefore,
    \[ w(x,t)=x+t\pm i\left(\frac{1}{3}x-3t+\nabla \varphi(x)\right)\in \text{Ker }(D\pm i\partial_t),\]
where $\varphi$ is an arbitrary harmonic function in $\Omega$.
\end{example}

Let $u_0\in C^{2,2}(\mathscr{C},\mathbb{C})$, let us define the following operator 
\begin{equation}\label{eq:operator-comp}
    \mathscr{U}[u_0](x,t):=i\int_0^t \nabla u_0(x,s)\, ds-i T_\Omega[\partial_t u_0](x,0).
\end{equation}
Observe that $\mathscr{U}[u_0]\in C^{1,1}(\mathscr{C}, \mathbb{C}^3)$ is purely vectorial whenever $u\in C^{2,2}(\mathscr{C},\mathbb{C})$. Moreover, if $u_0$ is a scalar solution of the wave equation \eqref{eq:wave-equation}, then $u_0\pm \mathscr{U}[u_0]$ belongs to the kernel of the parabolic Dirac operator \eqref{eq:parabolic-Dirac}. Indeed,
\begin{proposition}\label{prop:harmonic-conjugates}
    Let $u_0:\mathscr{C}\to\C$ in $C^{2,2}(\mathscr{C},\mathbb{C})$. If $u_0\in\text{Ker }(-\Delta+\partial_{tt}),$ then $u_0\pm\mathscr{U}[u_0]\in\text{Ker }(D\pm i\partial_t)$. Moreover, for any scalar harmonic function independent of time $h$, we have that $u_0\pm\mathscr{U}[u_0]+\nabla h\in\text{Ker}(D\pm i\partial_t)$.
\end{proposition}

\begin{proof} The proof is based on a direct computation
    \begin{align*}
        (D\pm i\partial_t)\left(u_0\pm \mathscr{U}[u_0]\right)(x,t)&=(D\pm i\partial_t)\left(u_0(x,t)\pm i\left[\int_0^t\nabla u_0(x,s)\, ds-T_\Omega[\partial_tu_0](x,0)\right]\right)\\
        &=\nabla u_0\pm i\int_0^t-\Delta u_0(x,s)\, ds\mp i \partial_tu_0(x,0)\pm i\partial_tu_0(x,t) - \nabla u_0\\
        &=\pm i\int_0^t-\partial_s^2 u_0(x,s)\, ds\mp i \partial_tu_0(x,0)\pm i\partial_tu_0(x,t)\\
        &=0.
    \end{align*}
\end{proof}
\begin{example}
Letting $u_0(x,t)=x_1^2+x_2^2+x_3^2+3t^2$ in Proposition \ref{prop:harmonic-conjugates}, then $u_0$ is a solution of the wave equation \eqref{eq:wave-equation}. We can compute $\mathscr{U}[u_0](x,t)$ as follows  
\begin{align*}
        \mathscr{U}[u_0](x,t)&=i\left[2\int_0^t  x \,ds-6T_\Omega[t](x,0)\right]=2itx.
\end{align*}
Consequently,
    \[w(x,t)=u_0\pm\mathscr{U}[u_0]=x_1^2+x_2^2+x_3^2+3t^2 \pm 2itx\in \text{Ker }(D\pm i\partial_t).\]
\end{example}

\section{Inverse of the parabolic Dirac operator $D\pm i\partial_t$.}\label{sec:inverse_dirac_operator}
One advantage of the factorization \eqref{eq:factorization-wave} is that fundamental solutions of the first-order operators $D\pm i\partial_t$ can be derived directly from a fundamental solution of the wave operator. However, after an exhaustive review of the literature, we opt against using the classical fundamental solutions of the wave operator; instead, we propose to work with a full characterization of the image of a fundamental solution of $D\pm i\partial_t$ under the classical Fourier transform, and we find an unexpected relation with the quaternionic exponential function.

Let us consider $\Phi_{\pm}\in L^1(\R^3\times(0,\infty))$ be a fundamental solution of the parabolic Dirac operator $D\pm i\partial_t$, i.e.
\begin{equation}\label{eq:homogeneous}
    (D\pm i\partial_t)\Phi_\pm(x,t)=0, \quad x\ne 0,\,  t\ne 0.
\end{equation}
Applying the Fourier transform defined in \eqref{eq:Fourier-space-time} to \eqref{eq:homogeneous} and using the property \eqref{eq:action-gradient}, we get that
\begin{equation}\label{eq:Fourier-exponential}
    2\pi ik\fourier{\Phi_\pm}(k,t)\pm i\partial_t\fourier{\Phi_\pm}(k,t)=0.
\end{equation}
Using that $\overline{k}=-k$, it is easy to verify that a solution of \eqref{eq:Fourier-exponential} is
\begin{equation}\label{eq:fundamental}
    \fourier{\Phi_\pm}(k,t)=\expo{\mp 2\pi kt}=\cos\left(2\pi|k|t\right)\mp\frac{k}{|k|}\sin\left(2\pi|k|t\right),
\end{equation}
where the right-hand side is a quaternionic-valued exponential function defined for all $(x,t)\in \mathbb{R}^3\times \mathbb{R}$, see for instance \cite[Theorem~11.19]{GuHaSpr2008}.
In particular, we consider $t>0$, we additionally suppose that $\fourier{\Phi_\pm}(k,t)=0$ for all $t<0$.

So the next natural step is to find $\fourierinv{\fourier{\Phi_\pm}}$, but this can be rather difficult, so instead we should use the properties \eqref{eq:definition-Fourier-inverse}-\eqref{eq:convolution} of the Fourier transform and the complete characterization of $\fourier{\Phi_{\pm}}$ given in \eqref{eq:fundamental}. Observe that
\begin{align*}
(\Phi_\pm*_{(x,t)}w)(x,t)&=\int_{\R}\int_{\R^3}\Phi_\pm(x-y,t-s)w(y,s)\, dy ds\\
    &=\int_{\R}\int_{\R^3}f_s(x-y)g_s(y)\, dy ds\\
    &=\int_{\R} (f_s*g_s)(x)\, ds\\
    &=\int_{\R}\fourierinv{\fourier{f_s*g_s}}(x)\, ds\\
    &=\int_{\R}\fourierinv{\fourier{f_s}\fourier{g_s}}(x)\, ds\\
    &=\int_{\R}\int_{\R^3}\expo{2\pi i\inner{x,k}}\fourier{f_s}(k)\fourier{g_s}(k)\, dkds.\\
    &=\int_0^t\int_{\R^3}\expo{2\pi i\inner{x,k}}\fourier{\Phi_\pm}(k,t-s)\fourier{w}(k,s)\, dkds.\\
    &=\int_0^t\int_{\R^3}\expo{2\pi i\inner{x,k}}\expo{\mp 2\pi k(t-s)}\fourier{w}(k,s)\, dkds,
\end{align*}
where $f_s(x)=\Phi_\pm(x,t-s)$ and $g_s(x)=w(x,s).$ 

After the above discussion, we define the \textit{parabolic Teodorescu transform} acting on integrable biquaternionic functions $w \colon \mathbb{R}^3\times(0,\infty) \rightarrow \mathbb{B}$ as follows

\begin{align}\label{eq:parabolic-Teodorescu}
T_{\mathscr{C},\pm}[w](x,t):=\int_0^t\int_{\R^3}\expo{2\pi i\inner{x,k}}\expo{\mp 2\pi k(t-s)}\fourier{w}(k,s)\, dk\, ds, 
\end{align}
for all $(x,t)\in \mathbb{R}^3\times(0,\infty)$. From now on, when we work with the parabolic Teodorescu transform $T_{\mathscr{C},\pm}$ acting on functions in $L^2(\mathscr{C},\B)$, we will extend by zero on $(\mathbb{R}^3\setminus \overline{\Omega}) \times (0,\infty)$. Moreover, 
\begin{proposition}\label{prop:well-defined}
    Let $w\in L^2(\mathscr{C},\B)$. Then $T_{\mathscr{C},\pm}[w](x,t)$ is well defined for almost all $(x,t)\in \mathbb{R}^3\times(0,\infty)$.
\end{proposition}
\begin{proof}
If $w\in L^2(\mathscr{C},\B)$, then $w(\cdot,t_0)\in L^1(\Omega, \mathbb{B})$ and extending by zero on $\mathbb{R}^3\setminus \overline{\Omega}$, then $\fourier{w}$ is well-defined. And now given $t>0$ then for all $0<s<t$ we have the following

    \begin{align*}
        \int_0^t\int_{\R^3}|\expo{\mp 2\pi (t-s)k}\fourier{w}(k,s)|^2\, dk\, ds&\leq\int_0^t\int_{\R^3}|\fourier{w}(k,s)|^2\, dk\, ds\\
        &=\int_0^t\norm{\fourier{w}(\cdot,s)}_{L^2(\mathbb{R}^3)}^2\, ds\\
        &=\int_0^t\norm{{w}(\cdot,s)}_{L^2(\mathbb{R}^3)}^2,
    \end{align*}
where the last equality comes from Plancherel's theorem. But given that $\norm{\cdot}$ is continuous, then is bounded over $[0,t]$. Therefore, $\int_0^t\expo{\mp 2\pi (t-s)k}\fourier{w}(k,s)\, ds$ is in $L^2(\Omega,\B)$ meaning it has a Fourier inverse. Thus,
 \begin{align*}
        &\fourierinv{\int_0^t\expo{\mp2\pi(t-s)}\fourier{w}(k,s)\,ds}(x)\\
        &=\int_{\R^3}\expo{2\pi i\inner{x,k}}\int_0^t\expo{\mp 2\pi(t-s)k}\fourier{w}(k,s)\,ds\, dk\\
        &=\int_0^t\int_{\R^3}\expo{2\pi i\inner{x,k}}\expo{\mp 2\pi(t-s)k}\fourier{w}(k,s)\, dk\, ds.\\
        &=T_{\mathscr{C},\pm}[w](x,t),
    \end{align*}
is well-defined for almost all $(x,t)\in \mathbb{R}^3\times(0,\infty)$ as we desired.
\end{proof}
Now, we enunciate the main result of this Section, which establishes that up to a multiplicative factor, the operator $T_{\mathscr{C},\pm}$ defined in \eqref{eq:parabolic-Teodorescu} is a right inverse operator of the parabolic Dirac operator $D\pm i\partial_t$.
\begin{theorem}\label{th:right-inverse}
    Let $w\in L^2(\mathscr{C},\B)$ then the operator $T_{\mathscr{C},\pm}[-i\text{ } \cdot]$ is a right inverse of $D\pm i\partial_t$. That is,
    \begin{equation}\label{eq:inverse-parabolic}
        (D\pm i\partial_t) T_{\mathscr{C},\pm}[\mp i \,w](x,t)= w(x,t), \qquad\forall (x,t)\in\R^3\times (0,\infty).
    \end{equation}
\end{theorem}
\begin{proof}
First, by convention, we will extend $w$ by zero on $\mathbb{R}^3\setminus \overline{\Omega}$. Notice that
    \begin{align*}
        \partial_jT_{\mathscr{C},\pm}[w](x,t)&=\partial_j\int_0^t\int_{\R^3}\expo{2\pi i\inner{x,k}}\expo{\mp2\pi(t-s)k}\fourier{w}(k,s)\, dk\, ds\\
        &=\int_0^t\int_{\R^3}2\pi ik_j\expo{2\pi i\inner{x,k}}\expo{\mp 2\pi(t-s)k}\fourier{w}(k,s)\, dk\, ds.
    \end{align*}
    Then,
    \[DT_{\mathscr{C},\pm}[w](x,t)=\int_0^t\int_{\R^3}2\pi ik\expo{2\pi i\inner{x,k}}\expo{\mp 2\pi(t-s)k}\fourier{w}(k,s)\, dk\,ds.\]
    We now compute the partial derivative with respect to $t$, applying the Leibniz integral rule.
    \begin{align*}
        \partial_t T_{\mathscr{C},\pm}[w](x,t)&=\partial_t\int_0^t\int_{\R^3}\expo{2\pi i\inner{x,k}}\expo{\mp 2\pi(t-s)k}\fourier{w}(k,s)\,dk\,ds\\
        &=\int_{\R^3}\expo{2\pi i\inner{x,k}}\fourier{w}(k,t)\, dk\\   &+\int_0^t\int_{\R^3}\partial_t\left[\expo{2\pi i\inner{x,k}}\expo{\mp 2\pi(t-s)k}\right]\fourier{w}(k,s)\, dk\,ds\\
        &=
        w(x,t)\mp \int_0^t\int_{\R^3}2\pi k\expo{2\pi i\inner{x,k}}\expo{\mp 2\pi(t-s)k}\fourier{w}(k,s)\,dk\, ds\\
        &=w(x,t)\pm iD T_{\mathscr{C},\pm}[w](x,t)(x,t).
    \end{align*}
    Therefore, we can easily verify that 
    $$(D\pm i\partial_t)T_{\mathscr{C},\pm}[w](x,t)=\pm iw(x,t), \qquad \forall (x,t)\in \R^3\times (0,\infty),$$
which in turn implies \eqref{eq:inverse-parabolic} as we desired.
\end{proof}
A straightforward consequence of Theorem \ref{th:right-inverse} and of the factorization \eqref{eq:factorization-wave} is 
\begin{corollary}\label{cor:compatibilty}
Under the same hypothesis than Theorem \ref{th:right-inverse}. The scalar part of $T_{\mathscr{C},\pm}[{\mp iw}]$ is solution of the wave equation \eqref{eq:wave-equation} if and only if 
\begin{align}\label{eq:compatibility-condition}
\div \vec w\pm i\partial_tw_0=0. 
\end{align}
\end{corollary}
\begin{proof}
Just notice the following equalities
    \begin{align*}
        (-\Delta+\partial_{tt}) T_{\mathscr{C},\pm}[{\mp iw}]&=(D\mp i\partial_t) (D\pm i\partial_t) T_{\mathscr{C},\pm}[{\mp iw}]\\
        &=(D\mp i\partial_t)w\\
        &=(D\mp i\partial_t)(w_0+\vec{w})\\
        &=-\div{\vec w}\mp i\partial_t w_0+\nabla w_0+\curl{\vec w}\mp i\partial_t\vec w.
    \end{align*}
Taking the scalar part, it yields the desired result.
\end{proof}
Observe that in the time-independent case, the condition \eqref{eq:compatibility-condition} reduces to the class of solenoidal vector fields \cite[Prop. 3.1]{DelPor2017} required in the construction of the solution of the inhomogeneous div-curl system provided in \cite{DelPor2017}.

\section{Solution of Maxwell's equations} \label{sec:rel_maxwell_equations}
A historical account of how Maxwell's equations were discovered can be found in \cite{huray2009maxwell,shapiro1973history}, while a rigorous and self-contained introduction to these four fundamental laws of electricity and magnetism is presented in \cite{jackson1977classical}. These laws (Gauss's law for electricity, the absence of free magnetic monopoles, Faraday's law of induction, and Ampère's law) govern the behavior of classical electromagnetic fields. Throughout this work, we consider Maxwell's equations in Gaussian units
\begin{equation}\label{eq:Maxwell}
    \begin{cases}
        \div\vec{E}&=4\pi\rho,\\
        \div\vec{B}&=0,\\
        \curl\vec{E}&=-\partial_t\vec{B},\\
        \curl\vec{B}&=4\pi\vec{\jmath}+\partial_t\vec{E},
    \end{cases}
\end{equation}
where $\rho$ is the electric charge density, $\vec\jmath$ is the current density, $\vec{E}$ is the electric field, and $\vec B$ is the magnetic field, respectively. Notice that these equations are written in units where the speed of light is one, meaning $c=1.$

It is crucial to note the compatibility condition linking the divergence 
of the current density to the time derivative of the electric charge density
\begin{equation}\label{eq:compatibility}
    \operatorname{div }\vec{\jmath}+\partial_t\rho=0,
\end{equation}
which expresses the \textit{law of conservation of electric charge}. 
This equation is not an external constraint but an intrinsic consistency 
condition of Maxwell's system. In integral form, it states that the rate of change of total charge 
enclosed in any volume equals the net inward flux of current through its 
boundary. From the perspective 
of PDE analysis, \eqref{eq:compatibility} is a necessary condition for the 
system to be well-posed: the source terms $\rho$ and $\vec{\jmath}$ cannot 
be prescribed independently. This constraint will play a central role in 
Theorem \ref{solutions:for:maxwell:equations}, where it enables a reformulation 
of Maxwell's equations admitting purely vectorial solutions.

The following Proposition, established in \cite[Sec.\ 3.1.1]{kravchenko2003applied} for Maxwell's equations 
in a homogeneous and isotropic medium, will be central to our analysis.
\begin{proposition}\label{eq:equivalence-Maxwell}
    Let $\rho$ and $\vec\jmath$ be continuously differentiable functions with domain $\Omega\times[0,\infty)$ and values in $\R$ and $\R^3$, respectively. Then $(\vec E, \vec B)$ is a solution of the Maxwell system \eqref{eq:Maxwell}  if and only if the biquaternionic function $\vec \varphi=\vec E+i\vec B$ satisfies 
    \begin{equation}\label{eq:system-equivalent}
            (D-i\partial_t)\vec \varphi=-4\pi(\rho-i\:\vec \jmath).
        \end{equation}
\end{proposition}

\begin{proof}
We proceed with a direct verification. 
\begin{align*}
        (D-i\partial_t)\vec \varphi=(D-i\partial_t)(\vec E+i\vec B)&=-\div(\vec E)+\curl(\vec E)-i\partial_t\vec E-i\div(\vec B)+i\curl(\vec B)+\partial_t\vec B\\
        &=-4\pi(\rho-i\:\vec\jmath).
\end{align*}
Separating into the scalar real, scalar imaginary, vector real, and vector imaginary parts, we obtain that \eqref{eq:system-equivalent} and \eqref{eq:Maxwell} are equivalent systems.
\end{proof}

Now, combining the previous equivalence with the ideas of construction of purely vectorial solutions of generalized div-curl systems from \cite{DelPor2017, DelPor2018, DelKrav2019, DelMac2021, delgado2023biquaternionic}, we provide an explicit solution of Maxwell's system \eqref{eq:Maxwell}.
\begin{theorem}\label{solutions:for:maxwell:equations}
    Let $\rho$ and $\vec\jmath$ be continuously differentiable functions with domain $\Omega\times[0,\infty)$ and values in $\R$ and $\R^3$, respectively. Suppose that $\rho$ and $\vec j$ satisfy the compatibility condition \eqref{eq:compatibility}. Then a solution $(\vec E, \vec B)$ of the Maxwell's system \eqref{eq:Maxwell} is as follows
        \begin{align*}
        \vec E(x,t)&=\text{Re }\left(\Vec\left(\int_0^t\int_{\R^3}\expo{2\pi i\inner{x,k}}\expo{2\pi(t-s)k}\fourier{-4\pi(\rho-i\:\vec\jmath)}(k,s)\, dkds\right)\right.\\
        &-\text{Im }\left\{\int_0^t \nabla_{x} u_0(x,s)\, ds-T_\Omega[\partial_t u_0](x,0) \right\}+\nabla h_1(x),
    \end{align*}
    \begin{align*}
        \vec B(x,t)&=\text{Im }\left(\Vec\left(\int_0^t\int_{\R^3}\expo{2\pi i\inner{x,k}}\expo{2\pi(t-s)k}\fourier{-4\pi(\rho-i\:\vec\jmath)}\,dkds\right)\right.\\
        &+\text{Re }\left\{\int_0^t \nabla _x u_0(x,s)\, ds-T_\Omega[\partial_t u_0](x,0)\right\}+\nabla h_2(x),
    \end{align*}
where $u_0:=\text{Sc }\Twm{-4\pi i(\rho-i\:\vec\jmath)}$, $h_1, h_2$ are harmonic functions independent of time.
\end{theorem}

\begin{proof}
Without loss of generality, we extend $\rho$ and $\vec\jmath$ as zero scalar and vector functions whenever $x\in \R^3\setminus \overline{\Omega}$. By Theorem \ref{th:right-inverse}, we have that
\begin{equation*}
    (D-i\partial_t)\Twm{-4\pi i(\rho-i\:\vec\jmath)}=-4\pi(\rho-i\vec j). 
\end{equation*}
By the compatibility condition \eqref{eq:compatibility} and Corollary \ref{cor:compatibilty}, we have that $u_0:=\text{Sc }\Twm{-4\pi i(\rho-i\:\vec\jmath)}$ is a solution of the wave equation \eqref{eq:wave-equation}. As a consequence, by Proposition \ref{prop:harmonic-conjugates}, we have that $u_0-\mathscr{U}[u_0]$ belongs to the kernel of $D-i\partial_t$. Thus, adding this biquaternionic function, we do not affect the system \eqref{eq:system-equivalent}. That is, \begin{align*}
\vec \varphi=\Twm{-4\pi i(\rho-i\:\vec\jmath)} -(u_0-\mathscr{U}[u_0])=\text{Vec } \Twm{-4\pi i(\rho-i\:\vec\jmath)}+\mathscr{U}[u_0]
\end{align*}
is a purely vectorial solution of \eqref{eq:system-equivalent}. Finally, by Proposition \ref{eq:equivalence-Maxwell} we get that $\vec E:=\text{Re }\vec \varphi$, $\vec B:=\text{Im }\vec \varphi$ solve Maxwell's system \eqref{eq:Maxwell} as we desired.
\end{proof}

\subsection{Maxwell's equations in homogeneous media}\label{subsec:maxwell_homogeneous}
Similarly to what we did in Section \ref{sec:inverse_dirac_operator}, we analyze the slightly modified parabolic Dirac operator $D\pm i\lambda\partial_t$, where $\lambda\in \mathbb{R}$. We look for a fundamental solution in $L^1(\R^3\times(0,\infty))$ to the equation
\[(D\pm i\lambda\partial_t)\Phi_{\lambda,\pm}(x,t)=0.\]
Using the Fourier transform, we similarly get that the quaternionic exponential function $\fourier{\Phi_{\lambda,\pm}}(x,t)=\expo{\mp 2\pi \lambda^{-1} tk}$ solves
\[2\pi k\fourier{\Phi_{\lambda,\pm}}(x,t)\pm i\lambda\partial_t\fourier{\Phi_{\lambda,\pm}}(x,t)=0.\]
Consequently, we define the \textit{$\lambda$ parabolic Teodorescu transform} as follows
\begin{equation}\label{eq:lambda-parabolic}
    T_{\mathscr{C},\pm,\lambda}[w](x,t):=\int_0^t\int_{\R^3} \expo{2\pi i\inner{x,k}}\expo{\mp 2\pi \lambda^{-1}(t-s)k}\fourier{w}(k,s)\, dsdk,
\end{equation}
for all $(x,t)\in \mathbb{R}^3\times(0,\infty)$. Similarly to Proposition \ref{prop:well-defined} and Theorem \ref{th:right-inverse}, the operator $T_{\mathscr{C},\pm,\lambda}$ is well-defined and serves as a right inverse of $D \pm i\lambda\partial_t$ on $\R^3\times (0,\infty)$. Moreover,

\begin{corollary}
    Let $w:\mathscr{C}\to\B$ be a differentiable function. Then 
    $$(D\pm i\lambda \partial_t)w=0\iff\left\{\begin{array}{rcl}
        -\div(\vec{u})&=&\pm \lambda\partial_t v_0,\\
        \nabla u_0+\curl(\vec{u})&=&\pm \lambda\partial_t\vec{v},\\
        -\div(\vec{v})&=&\mp \lambda\partial_t u_0,\\
        \nabla v_0+\curl(\vec{v})&=&\mp \lambda\partial_t\vec{u},\\
    \end{array}\right.$$
    where
    $$w=u+iv=(u_0+\vec{u})+i(v_0+\vec{v}).$$
\end{corollary}

Let us consider Maxwell's equations in a homogeneous and isotropic medium in SI units
\begin{equation}\label{eq:Maxwell-general}
    \left\{
        \begin{array}{ccc}
            \div{\vec E}&=&\dfrac{\rho}{\epsilon},\\
            \div{\vec B}&=&0,\\
            \curl{\vec E}&=&-\partial_t \vec B,\\
            \curl{\vec B}&=&\mu\left(\vec\jmath +\epsilon\partial_t\vec E\right),
        \end{array}
    \right. 
\end{equation}
where $\epsilon$ is the absolute permittivity and $\mu$ is the absolute permeability. From now on, 
we will suppose that $\epsilon, \mu\in \mathbb{R}$. Following \cite[Section~3.1.1]{kravchenko2003applied}, the author show that 
the operator $c^{-1}\partial_t\pm iD$ factorizes the wave operator as follows
\begin{equation*}
 \frac{1}{c^2}\partial_t-\Delta=\left(\frac{1}{c}\partial_t+iD\right)\left(\frac{1}{c}\partial_t-iD\right),
\end{equation*}
where $c$ is the speed of propagation of electromagnetic waves in the medium, given by $c=\frac{1}{\sqrt{\epsilon\mu}}$. Moreover, $(\vec E, \vec B)$ solves \eqref{eq:Maxwell-general} if and only if the biquaternionic function $\sqrt{\epsilon}\vec{E}+i\sqrt{\mu}\vec B$ satisfies
\begin{equation*}
    \left(\frac{1}{c}\partial_t+iD\right)(\sqrt{\epsilon}\vec{E}+i\sqrt{\mu}\vec B)=-\left(\sqrt{\mu}\;\vec\jmath+\frac{i\rho}{\sqrt{\epsilon}}\right).
\end{equation*}

Or equivalently, if $\sqrt{\epsilon}\vec{E}+i\sqrt{\mu}\vec B$ satisfies
\begin{equation}\label{eq:equi-general}
    \left(D-\frac{i}{c}\partial_t\right)\left(\sqrt{\epsilon}\vec{E}+i\sqrt{\mu}\vec B\right)=-\left(\frac{\rho}{\sqrt{\epsilon}}-i\sqrt{\mu}\;\vec\jmath\right).
\end{equation}
As in the previous results, solving \eqref{eq:Maxwell-general} is equivalent to finding a purely vectorial solution to \eqref{eq:equi-general}. Again, the proof relies on the properties of the Fourier transform, the $\lambda$ parabolic Teodorescu transform $T_{\mathscr{C},\pm,\lambda}$ defined in \eqref{eq:lambda-parabolic}, and taking $\lambda=\sqrt{\epsilon\mu}$. For the completion process illustrated in Proposition \ref{prop:harmonic-conjugates}, in this case, it is necessary to do a re-scaling with the constant $\lambda$, recall the operator $\mathscr{U}$ defined in \eqref{eq:operator-comp}, then

\begin{corollary}
      Let $u_0:\mathscr{C}\to\C$ in $C^{2,2}(\mathscr{C},\mathbb{C})$. If $u_0\in\text{Ker }(-\Delta+\lambda^2\partial_{tt}),$ then $u_0\pm\frac{1}{\lambda}\mathscr{U}[u_0]\in\text{Ker }(D\pm i\lambda \partial_t)$. Moreover, for any scalar harmonic function independent of time $h$, we have that $u_0\pm \frac{1}{\lambda}\mathscr{U}[u_0]+\nabla h\in\text{Ker}(D\pm i\lambda \partial_t)$.  
\end{corollary}

The novelty in this Subsection \ref{subsec:maxwell_homogeneous} is to provide an explicit solution of Maxwell's system \eqref{eq:Maxwell-general}.

\begin{theorem}\label{th:general}
 Let $\rho$ and $\vec\jmath$ be continuously differentiable functions with domain $\Omega\times[0,\infty)$ and values in $\R$ and $\R^3$, respectively. Suppose that $\rho$ and $\vec j$ satisfy the compatibility condition 
\begin{equation}
        \div\vec\jmath+\sqrt{\epsilon\mu}\partial_t\rho=0.
\end{equation}
Let $\lambda=\sqrt{\epsilon\mu}$. Then a solution $(\vec E, \vec B)$ of the Maxwell's system \eqref{eq:Maxwell-general} is as follows
       \begin{align*}
        \vec E(x,t)&=-\text{Re }\left(\Vec\left(\int_0^t\int_{\R^3}\expo{2\pi i\inner{x,k}}\expo{2\pi\lambda^{-1}(t-s)k}\fourier{\frac{\rho}{\sqrt{\epsilon}}-i\sqrt{\mu}\;\vec j}(k,s)\, dkds\right)\right.\\
        &+\lambda^{-1}\text{Im }\left\{\int_0^t \nabla_{x} u_0(x,s)\, ds-T_\Omega[\partial_t u_0](x,0) \right\}+\nabla h_1(x),
    \end{align*}
    \begin{align*}
        \vec B(x,t)&=-\text{Im }\left(\Vec\left(\int_0^t\int_{\R^3}\expo{2\pi i\inner{x,k}}\expo{2\pi\lambda^{-1}(t-s)k}\fourier{\frac{\rho}{\sqrt{\epsilon}}-i\sqrt{\mu}\;\vec j}(k,s)\,dkds\right)\right.\\
        &-\lambda^{-1}\text{Re }\left\{\int_0^t \nabla _x u_0(x,s)\, ds-T_\Omega[\partial_t u_0](x,0)\right\}+\nabla h_2(x),
    \end{align*}
where $u_0:=\text{Sc }\Twm{\frac{\rho}{\sqrt{\epsilon}}-i\sqrt{\mu}\;\vec j}$, $h_1, h_2$ are harmonic functions independent of time.
\end{theorem}

\section{Conclusions}
In this work, we analyzed the parabolic Dirac operator $D \pm i\partial_t$ and characterized its kernel via generalized div-curl systems and Cauchy-Riemann-type relations between the real and imaginary parts. Using Fourier analysis, we constructed a right inverse operator and applied it to obtain purely vectorial solutions for the time-dependent Maxwell system (see Theorems \ref{solutions:for:maxwell:equations} and \ref{th:general}).

These results highlight the synergy of Fourier transforms and quaternionic analysis for tackling complex PDE systems in electrodynamics. Future directions include extensions to nonlinear perturbations, relativistic frameworks, and numerical implementations.

\noindent {\bf Author Contributions}

All the authors wrote the main manuscript text. All authors reviewed the manuscript.\\

\noindent {\bf Data availability} 

No datasets were generated or analyzed during the current study.\\

\noindent {\bf Funding Declaration}

Funding: not applicable.

\section*{Declarations}
{\bf Conflict of interest}: The authors declare no conflict of interest.
\printbibliography
\bigskip
\small{
\noindent \text{Aarón Guillén-Villalobos} \\
Universidad de Guadalajara, Enrique Díaz de León 1144, Lagos de Moreno, Jalisco, Mexico \\
\texttt{aaron.guillen2541@alumnos.udg.mx} \\
ORCiD: 0000-0001-8605-7021

\bigskip
\noindent \text{Briceyda B. Delgado} \\
INFOTEC, Centro de Investigación e Innovación en Tecnologías de la Información y Comunicación \\
Cto. Tecnopolo Sur No. 112, Pocitos, Aguascalientes 20326, Mexico \\
\texttt{briceyda.delgado@infotec.edu.mx} \\
ORCiD: 0000-0002-1280-4650

\bigskip
\noindent \text{Héctor Vargas Rodríguez} \\
Universidad de Guadalajara, Av. Juárez No. 976, Zona Centro, C.P. 44100, Guadalajara, Jalisco, Mexico \\
\texttt{hector.vrodriguez@academicos.udg.mx} \\
ORCiD: 0000-0003-1973-9852
}

\end{document}